\documentclass{aptpub}

\usepackage{graphicx,verbatim}

\newcommand\diag{\operatorname{diag}}
\newcommand{\one}{\mathbf{1}}
\newcommand{\given}{\,|\,}
\newcommand{\Ef}{{\mathbb{E}}}
\newcommand{\EE}[1]{\Ef\!\left(#1\right)}
\newcommand{\Pf}{{\mathbb{P}}}
\renewcommand{\P}[1]{\Pf\!\left(#1\right)}

\authornames{Allman, An\'e, Rhodes}
\shorttitle{Identifiability of a model of molecular evolution}

\begin{document}

\title{Identifiability of a Markovian model of\\
molecular evolution with Gamma-distributed rates}
\date{\today}

\authorone[University of Alaska Fairbanks]{Elizabeth S.~Allman}
\addressone{Department of Mathematics and Statistics, University of Alaska Fairbanks, PO Box 756660, Fairbanks, AK 99775}
\emailone{e.allman@uaf.edu}
\authortwo[University of Wisconsin Madison]{C\'ecile~An\'e}
\addresstwo{Department of Statistics, University of Wisconsin Madison, Medical Science
Center, 1300 University Ave., Madison, WI 53706}
\emailtwo{ane@stat.wisc.edu}
\authorthree[University of Alaska Fairbanks]{John A.~Rhodes}
\addressthree{Department of Mathematics and Statistics, University of Alaska Fairbanks, PO Box 756660, Fairbanks, AK 99775}
\emailthree{j.rhodes@uaf.edu}

\vfil

\begin{abstract}
Inference of evolutionary trees and rates from biological sequences
is commonly performed using continuous-time Markov models of
character change. The Markov process evolves along an unknown tree
while observations arise only from the tips of the tree. Rate
heterogeneity is present in most real data sets and is accounted for
by the use of flexible mixture models where each site is allowed its
own rate. Very little has been rigorously established concerning the
identifiability of the models currently in common use in data
analysis, although non-identifiability was proven for a
semi-parametric model and an incorrect proof of identifiability was
published for a general parametric model (GTR+$\Gamma$+I). Here we
prove that one of the most widely used models (GTR+$\Gamma$) is
identifiable for generic parameters, and for all parameter choices
in the case of 4-state (DNA) models. This is the first proof of
identifiability of a phylogenetic model with a continuous
distribution of rates.
\end{abstract}

{

}

\section{Introduction}

A central goal of molecular phylogenetics is to infer evolutionary
trees from DNA or protein sequences. Such sequence data come from
extant species at the tips of the tree -- the tree of life -- while
the topology of the tree relating these species is unknown.
Inferring this tree helps us understand the evolutionary
relationships between sequences.

Phylogenetic data analysis is often performed using Markovian models
of evolution: Mutations occur along the branches of the tree under a
finite-state Markov process. There is ample evidence that some
places in the genome undergo mutations at a high rate, while other
loci evolve very slowly, perhaps due to some functional constraint.
Such \emph{rate variation} occurs at all spatial scales, across
genes as well as across sites within genes. In performing inference,
this heterogeneity is accounted for by the use of flexible mixture
models where each site is allowed its own rate according to a rate
distribution $\mu$. In the context of molecular phylogenetics, the
use of a parametric family for $\mu$ is generally considered both
advantageous and sufficiently flexible.

\smallskip

The question of identifiability for such a rate-variation model is a
fundamental one, as standard proofs of consistency of statistical
inference methods begin by establishing identifiability. Without
identifiability, inference of some or all model parameters may be
unjustified. However, since phylogenetic data is gathered only from
the tips of the tree, understanding when one has identifiability of
the tree topology and other parameters for phylogenetic models poses
substantial mathematical challenges. Indeed, it has been shown that
the tree and model parameters are \emph{not} identifiable if the
distribution of rates $\mu$ is too general, even when the Markovian
mutation model is quite simple \cite{SSH94}.

The most commonly used phylogenetic model is a \emph{general
time-reversible} (GTR) Markovian mutation model along with a Gamma
distribution family ($\Gamma$) for $\mu$. For more flexibility, a
class of invariable sites (I) can be added by allowing $\mu$ to be
the mixture of a Gamma distribution with an atom at $0$
\cite{Fel04}. Numerous studies have shown that the addition to the
GTR model of rate heterogeneity through $\Gamma$, I, or both,  can
considerably improve fit to data at the expense of only a few
additional parameters. In fact, when model selection procedures are
performed, the GTR+$\Gamma$+I model is preferred in most studies.
These stochastic models are the basis of hundreds of publications
every year in the biological sciences
--- over 40 in \emph{Systematic Biology} alone in 2006.
%
%
Their impact is immense in the fields of evolutionary biology,
ecology, conservation biology, and biogeography, as well as in
medicine, where, for example, they appear in the study of the
evolution of infectious diseases such as HIV and influenza viruses.

\smallskip

The main result claimed in the widely-cited paper \cite{Rog01}
is the following:

\smallskip

\noindent \emph{The $4$-base (DNA) GTR+$\Gamma$+I model, with
unknown mixing parameter and $\Gamma$ shape parameter, is
identifiable from the joint distributions of pairs of taxa.}

\smallskip
\noindent However, the proof given in \cite{Rog01} of this statement
is flawed; in fact, two gaps occur in the argument. The first gap is
in the use of an unjustified claim concerning graphs of the sort
exemplified by Figure 3 of that paper. As this claim plays a crucial
role in the entire argument, the statement above remains unproven.

The second gap, though less sweeping in its impact, is still
significant. Assuming the unjustified graphical claim mentioned
above could be proved, the argument of \cite{Rog01} still uses an
assumption that the eigenvalues of the GTR rate matrix be distinct.
While this is true for \emph{generic} GTR parameters, there are
exceptions, including the well-known Jukes-Cantor and Kimura
2-parameter models \cite{Fel04}. Without substantial additional
arguments, the reasoning given in \cite{Rog01} cannot prove
identifiability in all cases.

Furthermore, bridging either of the gaps in \cite{Rog01} is not a
trivial matter. Though we suspect that Rogers' statement of
identifiability is correct, at least for generic parameters, we have
not been able to establish it by his methods. For further exposition
on the nature of the gaps, see the Appendix.

\medskip

In this paper, we consider only the GTR+$\Gamma$ model, but for
characters with any number $\kappa\ge 2$ states, where the case
$\kappa=4$ corresponds to DNA sequences. Our main result is the
following:

\begin{theorem}\label{thm:id}
The $\kappa$-state GTR+$\Gamma$ model is identifiable from the joint
distributions of triples of taxa for generic parameters on any tree
with $3$ or more taxa.
Moreover, when $\kappa=4$ the model is identifiable for all
parameters.
\end{theorem}

The term `generic' here means for those GTR state distributions and
rate matrices which do not satisfy at least one of a collection of
equalities to be explicitly given in Theorem \ref{thm:genid}.
Consequently, the set of non-generic parameters is of Lebesgue
measure zero in the full parameter space.
Our arguments are quite different from those attempted in \cite{Rog01}.
We combine arguments from algebra, algebraic geometry and analysis.

\medskip

We believe this paper presents the first correct proof of
identifiability for any model with a continuous distribution $\mu$
of rates across sites that is not fully known. The
non-identifiability of some models with more freely-varying rate
distributions of rates across sites was established in \cite{SSH94}.
That paper also showed identifiability of rate-across-sites models
built upon certain group-based models provided the rate distribution
$\mu$ is completely known. More recently, \cite{ARidtree} proved
that tree topologies are identifiable for generic parameters in
rather general mixture models with a small number of classes. That
result specializes to give the identifiability of trees for the
$\kappa$-state GTR models with at most $\kappa-1$ rates-across-sites
classes, including the GTR+I model. Identifiability of numerical
model parameters for GTR+I is further explored in \cite{ARGMI}.
There have also been a number of recent works dealing with
non-identifiability of mixture models which are not of the
rates-across-sites type; these include
\cite{StefVig,StefVig1,MatSt,MatMoSt}.

\medskip

In Section \ref{sec:prelim} we define the GTR+$\Gamma$ model,
introduce notation, and reduce Theorem \ref{thm:id} to the case of a
3-taxon tree. In Section \ref{sec:alg}, we use purely algebraic
arguments to determine from a joint distribution certain useful
quantities defined in terms of the model parameters. In Section
\ref{sec:gen}, in the generic case of certain algebraic expressions
not vanishing, an analytic argument uses these quantities to
identify the model parameters. Focusing on the important case of
$\kappa=4$ for the remainder of the paper, in Section
\ref{sec:badmods} we completely characterize the exceptional cases
of parameters not covered by our generic argument. Using this
additional information, in Section \ref{sec:badid} we establish
identifiability for these cases as well. Finally, Section
\ref{sec:open} briefly mentions several problems concerning
identifiability of phylogenetic models that remain open.

\section{Preliminaries}\label{sec:prelim}

\subsection{The GTR+rates-across-sites substitution model}

The $\kappa$-state across-site rate-variation model is parameterized
by:

\begin{enumerate}
\item An unrooted topological tree $T$, with all internal vertices of valence $\ge 3$,
and with leaves labeled by $a_1,a_2,\dots, a_n$. These labels
represent taxa, and the tree their evolutionary relationships.

\item A collection of edge lengths $t_e\ge 0$,
where $e$ ranges over the edges of $T$. We require $t_e>0$ for all
internal edges of the tree, but allow $t_e\ge 0$ for pendant edges,
provided no two taxa are total-edge-length-distance 0 apart. Thus if
an edge $e$ is pendant, the label on its leaf may represent either
an ancestral ($t_e=0$) or non-ancestral ($t_e>0$) taxon.

\item A distribution vector $\boldsymbol
\pi=(\pi_1,\dots,\pi_\kappa)$ with $\pi_i>0$, $\sum \pi_i=1$,
representing the frequencies of states occurring in biological
sequences at all vertices of $T$.

\item A $\kappa\times\kappa$ matrix $Q=(q_{ij})$, with $q_{ij}>0$ for $i\ne j$
and $\sum_j q_{ij}=0$ for each $i$, such that
$\diag(\boldsymbol\pi)Q$ is symmetric. $Q$ represents the
instantaneous substitution rates between states in a reversible
Markov process. We will also assume some normalization of $Q$ has
been imposed, for instance that $\diag(\boldsymbol\pi)Q$ has trace
$-1$.

Note that the symmetry and row summation conditions imply that
$\boldsymbol \pi$ is a left eigenvector of $Q$ with eigenvalue 0,
which in turn implies $\boldsymbol \pi$ is stationary under the
continuous-time process defined by $Q$.

\item A distribution $\mu$, with non-negative support and expectation
$\EE{\mu}=1$, describing the distribution of rates among sites. If a
site has rate parameter $r$, then its instantaneous substitution
rates will be given by $rQ$.

\end{enumerate}

Letting $[\kappa]=\{1,2,\dots,\kappa\}$ denote the states, the joint
distribution of states at the leaves of the tree $T$ which arises
from a rate-across-sites GTR model is computed as follows. For each
rate $r$ and edge $e$ of the tree, let $M_{e,r}=\exp(t_erQ)$. Then
with an arbitrary vertex $\rho$ of $T$ chosen as a root, let
\begin{equation}
P_r(i_1,\dots,i_n)= \sum_{(h_v)\in H} \left (
\boldsymbol{\pi}(h_\rho) \prod_{e}M_{e,r}(h_{s(e)},h_{f(e)})\right
),\label{eq:Pdef}
\end{equation}
where the product is taken over all edges $e$ of $T$ directed away
from $\rho$, edge $e$ has initial vertex $s(e)$ and final vertex
$f(e)$, and the sum is taken over the set $$H=H_{i_1i_2\dots
i_n}=\left \{ (h_v)_{v\in\operatorname{Vert}(T)} ~|~ h_v\in [\kappa]
\text{ if $v\ne a_j$},\ h_v=i_j \text{ if $v=a_j$}\right \}\subset
[\kappa]^{|\operatorname{Vert}(T)|}.$$ Thus $H$ represents the set
of all `histories' consistent with the specified states
$i_1,\dots,i_n$ at the leaves, and the $n$-dimensional table $P_r$
gives the joint distribution of states at the leaves given a site
has rate parameter $r$. Since the Markov process is reversible and
stationary on $\boldsymbol \pi$, this distribution is independent of
the choice of root $\rho$.

Finally, the joint distribution for the GTR+$\mu$ model is given by
the $n$-dimensional table
$$P=\int_r P_r d\mu(r).$$
The distribution for the GTR+$\Gamma$ model is given by additionally
specifying a parameter $\alpha>0$, with $\mu$ then specialized to be
the $\Gamma$-distribution with shape parameter $\alpha$ and mean 1,
{\it i.e.}, with scale parameter $\beta=1/\alpha$.

\subsection{Diagonalization of $Q$}\label{ss:diag}
The reversibility assumptions on a GTR model imply that
$\diag({\boldsymbol\pi}^{1/2})Q\diag({\boldsymbol\pi}^{-1/2})$
is symmetric, and that $Q$
can be represented as
$$Q=U\mathrm{diag}(0,\lambda_2,\lambda_3,\dots,\lambda_\kappa)U^{-1},$$
where the eigenvalues
of $Q$ satisfy
$0=\lambda_1>\lambda_2\geq\lambda_3\geq\dots\geq\lambda_\kappa$
%
\cite{hornJohnson85},
and $U$ is a real matrix of associated
eigenvectors satisfying the equivalent statements
\begin{equation}
\label{eq:orth2} UU^T = \diag(\boldsymbol \pi)^{-1},\ \
U^T\diag(\boldsymbol \pi)U=I.
\end{equation}
Furthermore, the first column of $U$ may be taken to
be the vector $\one$.

While the $\lambda_i$ are uniquely determined by these
considerations, in the case that all $\lambda_i$ are distinct the
matrix $U$ is determined only up to multiplication of its individual
columns by $\pm1$. If the $\lambda_i$ are not distinct, eigenspaces
are uniquely determined but $U$ is not.

Our method of determining $Q$ from a joint distribution will proceed
by determining eigenspaces (via $U$) and the $\lambda_i$ separately. Although the
non-uniqueness of $U$ will not matter for our arguments, the
normalization  determined by equations (\ref{eq:orth2}) will be used
to simplify our presentation.

\subsection{Moment generating function}

We also use the moment generating function (\emph{i.e.}, essentially
the Laplace transform) of the density function for the distribution
of rates in our model. As our algebraic arguments will apply to
arbitrary rate distributions, while our analytic arguments are
focused on $\Gamma$ distributions, we introduce notation for the
moment generating functions in both settings.

\begin{definition} For any fixed distribution $\mu$ of rates $r$, let $$L(u)=L_\mu(u)=\EE{e^{ru}}
$$ for $-\infty<u\le 0$,
denote the expectation of $e^{ru}$. In the special case of
$\Gamma$-distributed rates, with parameters $\alpha>0$ and
$\beta=1/\alpha$, let
$$L_\alpha(u)=L_{\Gamma,\alpha}=\EE{e^{ru}}
= \left (1-\frac u\alpha \right)^{-\alpha}.$$
\end{definition}

Note that $L$, and in particular $L_\alpha$, is an increasing
function throughout its domain.

\subsection{Reduction to 3-taxon case}
To prove Theorem \ref{thm:id}, it is sufficient to consider only the
case of 3-taxon trees.

\begin{figure}[h]
\begin{center}
\includegraphics{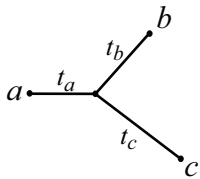}
\end{center}
\caption{The unique 3-taxon tree relating taxa $a$, $b$, and $c$,
with branch lengths $t_a$, $t_b$ and $t_c$.}\label{fig:triple}
\end{figure}

\begin{lemma}
If the statements of Theorem \ref{thm:id} holds for $3$-taxon trees,
then they also hold for $n$-taxon trees when $n>3$.
\end{lemma}
\begin{proof}
As the generic condition of Theorem 1 is a condition on $\boldsymbol
\pi$ and $Q$ (see Theorem \ref{thm:genid} below for a precise
statement), parameters on a $n$-taxon tree are generic if and only
if the induced parameters on all induced 3-taxon trees are generic.

If the model on 3-taxon trees is identifiable for certain
parameters, then from the joint distribution for a tree such as that
of Figure \ref{fig:triple}, we may determine $\alpha$, $Q$,
$\boldsymbol \pi$ and the 3 edge lengths $t_a,t_b,t_c$. Thus we may
determine the pairwise distances $t_a+t_b$, $t_a+t_c$, $t_b+t_c$
between the taxa. From an $n$-taxon distribution, by considering
marginalizations to 3 taxa we may thus determine $\alpha$, $Q$,
$\boldsymbol \pi$, and all pairwise distances between taxa. From all
pairwise distances, we may recover the topological tree and all edge
lengths by standard combinatorial arguments, as in \cite{MR2060009}.

\end{proof}

\section{Algebraic arguments}\label{sec:alg}

We now determine some information that we may obtain algebraically
from a joint distribution known to have arisen from the GTR+$\mu$
model on a tree $T$ relating 3 taxa. While in this paper we will
only apply the results to the GTR+$\Gamma$ model, we derive them at
their natural level of generality. We therefore denote the moment
generating function of the rate distribution by $L$, with its
dependence on $\mu$ left implicit.

As marginalizations of the joint distribution correspond to the
model on induced trees $T'$ with fewer taxa, we work with trees with
1, 2, or 3 leaves.

\medskip

If $T'$ has only 1 leaf, it is simply a single vertex, and the
distribution of states is therefore $\boldsymbol \pi$. Thus
$\boldsymbol \pi$ is identifiable from a joint distribution for 1 or
more taxa.

\medskip

If $T'$ has exactly 2 leaves, joined by an edge of length $t_e>0$,
then the joint distribution can be expressed as
$$P=\diag(\boldsymbol \pi)\EE{\exp(t_erQ)} =\diag(\boldsymbol
\pi)U\diag(L(\lambda_1 t_e),\dots,L(\lambda_\kappa t_e))U^{-1}.$$
Therefore, diagonalizing $\diag(\boldsymbol \pi)^{-1} P$ determines
the collection of $L(\lambda_i t_e)$ and the columns of $U$ up to
factors of $\pm 1$. Since $L$ is increasing, we may determine
individual $L(\lambda_i t_e)$ by the requirement that
\begin{equation}\label{eq:eigineq} 1=L(0)=L(\lambda_1
t_e)> L(\lambda_2 t_e)\ge \dots\ge L(\lambda_\kappa
t_e).\end{equation} When the $\lambda_i$ are distinct, this fixes an
ordering to the columns of $U$. Regardless, we simply make a fixed
choice of some $U$ consistent with the inequalities
(\ref{eq:eigineq}) and satisfying equations (\ref{eq:orth2}). We can
further require this choice of $U$ be made consistently for all
2-taxon marginalizations of the joint distribution. Thus for any
tree relating 2 or more taxa, we may determine the eigenspaces of
$Q$ via $U$ and the value $L(\lambda_i d_{jk})$ for each $i$ and
pair of taxa $a_j,a_k$, where $d_{jk}$ is the total edge-length
distance between $a_j$ and $a_k$.

\medskip

For $T$ with exactly 3 leaves, let  $a,b,c$ be the taxa labeling
them, with edge lengths as in Figure \ref{fig:triple}, and let $X_a,
X_b, X_c$ denote the character states at these taxa. As in
\cite{MR97k:92011}, denote by $P^{ab,\gamma}$ the square matrix
containing the probabilities
$$P^{ab,\gamma}(i,j) = \P{X_b=j,X_c=\gamma\given X_a=i},$$
which can be computed from the joint distribution. But
$$P^{ab,\gamma} = \EE{e^{rt_aQ}\diag\left({e^{rt_cQ}}_{\cdot
\gamma}\right)e^{rt_bQ}}
$$
where ${e^{rt_cQ}}_{\cdot \gamma}$ is the $\gamma^{th}$ column of
matrix $e^{rt_cQ}$, so
\begin{multline*}
U^{-1}P^{ab,\gamma}U =\\
\EE{\diag(e^{rt_a\lambda_1},\dots,e^{rt_a\lambda_\kappa})U^{-1}
  \diag\left({e^{rt_cQ}}_{\cdot \gamma}\right)U
  \diag(e^{rt_b\lambda_1},\dots,e^{rt_b\lambda_\kappa})}\;.
\end{multline*}
Note that the $j$th column of
$$\diag\left({e^{rt_cQ}}_{\cdot \gamma}\right)U$$ is the same as
the $\gamma$th column of
$$\diag\left(U_{\cdot j}\right)e^{rt_cQ}\;.$$
Thus when $(i,j)$ is fixed, the row vector formed by
$U^{-1}P^{ab,\gamma}U\,(i,j)$ for $\gamma=1,\dots,\kappa$ is
\begin{equation}\label{eq:step1}
\mu^{ij} \EE{e^{rt_a\lambda_i}e^{rt_b\lambda_j}e^{rt_cQ}}
\end{equation}
where $\mu^{ij}$ is the row vector with
\begin{equation}\label{eq:mu}
\mu^{ij}(k) = U^{-1}(i,k)U(k,j) = \pi(k)U(k,i)U(k,j)\,.
\end{equation}

Finally, multiplying \eqref{eq:step1} by $U$ on the right, and
setting $\nu^{ij} = \mu^{ij}U$, we see that the information brought
by the triple of taxa $\{a,b,c\}$ amounts to the knowledge of
$$\nu^{ij} \EE{e^{rt_a\lambda_i}e^{rt_b\lambda_j}
  \mbox{ diag}(e^{rt_c\lambda_1},\dots,e^{rt_c\lambda_\kappa})},$$
\emph{i.e.,} to the knowledge of each
$$\EE{e^{rt_a\lambda_i}e^{rt_b\lambda_j}e^{rt_c\lambda_k}}
=L(t_a\lambda_i+t_b\lambda_j+t_c\lambda_k)
$$
for which $\nu^{ij}(k)\neq 0$.

\smallskip

This motivates the following notation, where for conciseness we let
$U_{ij}=U(i,j)$: For $i,j,k\in[\kappa]$, let
$$\nu_{ijk} = \sum_l \pi_l U_{li}U_{lj}U_{lk}\,.$$
Note that while $\nu_{ijk}=\nu^{ij}(k)$, we prefer this new notation
since the value of $\nu_{ijk}$ is unchanged by permuting subscripts:
$$\nu_{ijk}=\nu_{ikj}=\nu_{jik}=\nu_{jki}=\nu_{kij}=\nu_{kji}.$$
Furthermore, since $\boldsymbol \pi$ can be determined from 1-taxon
marginalizations, and $U$ from 2-taxon marginalizations, from a
3-taxon distribution we may compute $\nu_{ijk}$ for all $i,j,k$.

In summary, we have shown the following:

\begin{proposition}\label{prop:triple}
From a distribution arising from the GTR+$\mu$ model on the
$3$-taxon tree of Figure \ref{fig:triple}, we may obtain the
following information:
\begin{enumerate}
\item $\boldsymbol \pi,$ from 1-marginalizations

\item \label{it:2} all matrices $U$ which diagonalize $Q$ as above, and for all $i$ the values
$$L(\lambda_i(t_a+t_b)),\
L(\lambda_i(t_a+t_c)),\ L(\lambda_i(t_b+t_c)),$$ from
2-marginalizations, and

\item \label{it:3} the values
$L(\lambda_i t_{a}+\lambda_j t_b +\lambda_k t_c)\text{ for all
$i,j,k$ such that $\nu_{ijk}\ne 0$}$ for some such choice of $U$.
\end{enumerate}
\end{proposition}

Note that (\ref{it:2}) can be obtained as a special case of
(\ref{it:3}) by taking $j=i,$ $k=1$, as it is easy to see
$\nu_{ii1}\ne0$. We shall also see that $\nu_{ij1}=0$ if $i\ne j$,
so certainly some of the $\nu_{ijk}$ can vanish.

\medskip

One might expect that for most choices of GTR parameters all the
$\nu_{ijk}\ne 0$ for $i,j,k>1$. Indeed, this is generally the case,
but for certain choices one or more of these $\nu_{ijk}$ can vanish.
The Jukes-Cantor and Kimura 2- and 3-parameter models provide simple
examples of this for $\kappa$=4: For these models, one may choose
$$\boldsymbol \pi =(1/4,\,1/4,\,1/4,\,1/4),\ \ U=\begin{pmatrix}
1&\phantom{-}1&\phantom{-}1&\phantom{-}1\\1&-1&\phantom{-}1&-1\\
1&\phantom{-}1&-1&-1\\1&-1&-1&\phantom{-}1\end{pmatrix},$$ and
$\nu_{ijk}\ne 0$ for $i,j,k>1$ only when $i,j,k$ are distinct. While
for the Jukes-Cantor and Kimura 2-parameter models one may make
other choices for $U$, one can show that these alternative choices
of $U$ do not lead to the recovery of any additional information.

 Nonetheless, for $\kappa\ge 3$ there is always some
genuine 3-taxon information available from a distribution, as we now
show. Although we do not need the following proposition for the
proof of Theorem \ref{thm:id}, the method of argument it introduces
underlies Section \ref{sec:badmods} below.

\begin{proposition}\label{prop:non0}
With $\kappa\ge 3$, for any choice of GTR parameters there exists at
least one triple $i,j,k>1$ with $\nu_{ijk}\ne 0$.
\end{proposition}

\begin{proof}
Suppose for all triples $i,j,k>1$,
\begin{equation}\label{eq:o1}\nu_{ijk}=\sum_l \pi_l U_{li}U_{lj}U_{lk}= 0.
\end{equation}
From equation (\ref{eq:orth2}) we also know that if $j\ne k$,  then
\begin{equation}\label{eq:o2}\nu_{1jk}=\sum_l \pi_l U_{lj}U_{lk}= 0.
\end{equation}

Both of these equations can be expressed more conveniently by
introducing the inner product
$$\langle x,y \rangle=x^T\diag(\boldsymbol \pi)y.$$ Then with $U_i$
being the $i$th column of $U$, and $W_{jk}$ being the vector whose
$l$th entry is the product $U_{lj}U_{lk}$, equations (\ref{eq:o1})
give the orthogonality statements

$$\langle U_i,W_{jk} \rangle=0, \text{ if $i,j,k>1$, }$$
while equations (\ref{eq:o2}) yield both
\begin{align*}
\langle U_1,W_{jk}\rangle&=0, \text{ if $j\ne k$,  and}\\
\langle U_j,U_{k}\rangle&=0, \text{ if $j\ne k$.}
\end{align*}
In particular, we see for $j,k>1$, $j\ne k$, that $W_{jk}$ is
orthogonal to all $U_i$, and so $W_{jk}=\mathbf 0$. Considering
individual entries of $W_{jk}$ gives that, for every $l$,
\begin{equation}
\label{eq:0prod}U_{lj}U_{lk}=0, \text{for all $j,k>1$, $j\ne k$. }
\end{equation}

Now note that for any $j>1$, the vector $U_j$ must have at least 2
non-zero entries. (This is simply because $U_j$ is a non-zero
vector, and  $\langle \mathbf 1,U_j\rangle=0$ since $U_1=\one$.) We
use this observation, together with equation (\ref{eq:0prod}), to
arrive at a contradiction.

First, without loss of generality, assume the first two entries of
$U_2$ are non-zero. Then by equation (\ref{eq:0prod}) the first two
entries of all the vectors $U_3, U_4, \dots$ must be 0. But then we
may assume the third and fourth entries of $U_3$ are non-zero, and
so the first 4 entries of $U_4, \dots$ are zero. For the 4-state DNA
model, this shows $U_4=\mathbf 0$, which is impossible.

More generally, for a $\kappa$-state model, we find $U_k=\mathbf 0$
as soon as $2(k-2)\ge \kappa$. Note that for $\kappa\ge 4$ this
happens for some value of $k\le\kappa$, thus contradicting that the
$U_k$ are non-zero. In the $\kappa=3$ case the same argument gives
that $U_3$ has only one non-zero entry, which is still a
contradiction, since $U_3$ is orthogonal to $U_1=\one$. Thus the
lemma is established for a $\kappa$-state model with $\kappa\ge 3$.
\end{proof}

For $\kappa=2$, the statement of Proposition \ref{prop:non0} does
not hold, as is shown by considering the 2-state symmetric model,
with
$$\boldsymbol \pi=(1/2,\,1/2),\text{ and } U=\begin{pmatrix}1&\phantom{-}1\\1&-1
\end{pmatrix}.$$ However, one can show
this is the only choice of $\boldsymbol \pi$ and $U$ for which
$\nu_{222}=0$.

\section{Identifiability for generic parameters}\label{sec:gen}

We now complete the proof of the first statement in Theorem
\ref{thm:id}, the identifiability of the GTR+$\Gamma$ model for
generic parameters, which is valid for all values of $\kappa\ge 2$.
As we now consider only $\Gamma$-distributed rates, we use the
specialized moment generating function $L_\alpha$ in our arguments.

More precisely, we will establish the following:

\begin{theorem}\label{thm:genid} For $\kappa\ge 2$, consider those GTR
parameters for which there exist some $i,j$, with $1<i\le j$, such
that $\nu_{ijj}\ne 0$. Then restricted to these parameters, the
GTR+$\Gamma$ model is identifiable on $3$-taxon trees.
\end{theorem}

\begin{remark}
Note that the conditions $\nu_{ijj}=0$ are polynomial in the entries
of $U$ and $\boldsymbol \pi$. Viewing the GTR model as parameterized
by those variables together with the $\lambda_i$, then the set of
points in parameter space for which $\nu_{ijj}=0$ for some $i,j$
with $1<i\le j$ forms a proper algebraic variety. Basic facts of
algebraic geometry then implies this set is of strictly lower
dimension than the full parameter space. A generic point in
parameter space therefore lies off this exceptional variety,  and
the exceptional points have Lebesgue measure zero in the full
parameter space.
\end{remark}

\begin{remark}
For $\kappa=2$, identifiability does not hold for the 3-taxon tree
if the generic condition that $\nu_{ijj}\ne 0$ for some $1<i\le j$
is dropped. Indeed, if $\nu_{222}=0$, then, as commented in the last
section, $\boldsymbol \pi$ and $U$ arise from the 2-state symmetric
model. Since there are only two eigenvalues of $Q$, $\lambda_1=0$
and $\lambda_2<0$, the second of these is determined by the
normalization of $Q$. As the proof of Proposition \ref{prop:triple}
indicates, the only additional information we may obtain from the
joint distribution is the three quantities
$$L_\alpha(\lambda_2(t_a+t_b)),\
L_\alpha(\lambda_2(t_a+t_c)),\ L_\alpha(\lambda_2(t_b+t_c)).$$ Since
these depend on four unknown parameters $\alpha,t_a,t_b,t_c$, it is
straightforward to see the parameter values are not uniquely
determined.
\end{remark}

\smallskip

Our proof of Theorem \ref{thm:genid} will depend on the following
technical lemma.

\begin{lemma}\label{lem:convex}
Suppose $c\ge a\ge d_1> 0$ and $c\ge b>d_2>0$. Then the equation
$$d_1^{-\beta}+d_2^{-\beta}-a^{-\beta}-b^{-\beta}-c^{-\beta}+1=0.$$
has at most one solution with $\beta>0.$
\end{lemma}

\begin{proof} The equation can be rewritten as
\begin{equation}\label{eq:rewrite}
\left ( \left (\frac{c}{d_1}\right)^\beta -\left (\frac {c}{a}\right
)^\beta \right ) + \left ( \left (\frac{c}{d_2}\right)^\beta -\left
( \frac {c}{b}\right )^\beta \right ) +\left (c^\beta -1\right ) =0
\end{equation}

Now a function $g(\beta)=r^\beta-s^\beta$ is strictly convex on
$\beta\ge 0$  provided $r> s\ge 1$, since $g''(\beta)>0$. If $r=s$,
then $g(\beta)=0$ is still convex. Thus when viewed as a function of
$\beta$ the first expression on the left side of equation
(\ref{eq:rewrite}) is convex, and the second expression is strictly
convex. Also, for any $r>0$ the function $h(\beta)=r^{\beta}-1$ is
convex, so the third expression in equation (\ref{eq:rewrite}) is
convex as well. Thus the sum of these three terms, the left side of
equation (\ref{eq:rewrite}), is a strictly convex function of
$\beta$.

But a strictly convex function of one variable can have at most two
zeros. Since the function defined by the left side of equation
(\ref{eq:rewrite}) has one zero at $\beta=0$, it therefore can have
at most one zero with $\beta>0$.
\end{proof}

\begin{proof}[Proof of Theorem \ref{thm:genid}]
For some $j\ge i>1$, we are given that $\nu_{ijj}\ne 0$. As
$\nu_{ijj}=\nu_{jij}$, by Proposition \ref{prop:triple} we may
determine the values
\begin{align*}
D_{ijj}&=L_\alpha(\lambda_i t_a+\lambda_j t_b+\lambda_j t_c),\\
D_{jij}&=L_\alpha(\lambda_j t_a+\lambda_i t_b+\lambda_j t_c),
\end{align*}
as well as
\begin{align*}
C_k&=L_\alpha(\lambda_k(t_a+t_b)),\\
B_k&=L_\alpha(\lambda_k(t_a+t_c)),\\
A_k&=L_\alpha(\lambda_k(t_b+t_c))
\end{align*} for $k=1,\dots,\kappa$.

Since $L_\alpha$ is increasing, for any $k>1$ we can use the values
of $C_k,B_k$ to determine which of $t_b$ and $t_c$ is larger.
Proceeding similarly, we may determine the relative ranking of
$t_a$, $t_b$, and $t_c$. Without loss of generality, we therefore
assume
$$0\le t_a\le t_b\le t_c$$
for the remainder of this proof. Note however that if $t_a=0$, then
$t_b>0$, by our assumption on model parameters that no two taxa be
total-edge-length-distance 0 apart.

\smallskip

Observe that $$L_\alpha^{-1} (D_{ijj})+L_\alpha^{-1}(D_{jij})=
L_\alpha^{-1}(A_j) + L_\alpha^{-1}(B_j)+ L_\alpha^{-1}(C_i),$$ or,
using the formula for $L_\alpha$ and letting $\beta=1/\alpha $,
\begin{equation}\label{eq:ddabc}
D_{ijj}^{-\beta}+ D_{jij}^{-\beta}- A_j^{-\beta} -B_j^{-\beta}-
C_i^{-\beta}+1=0.
\end{equation}

Since $j\ge i>1$, we have that $\lambda_j\le \lambda_i<0.$ Because
$L_\alpha$ is an increasing function, and $0\le t_a\le t_b\le t_c,$
with $t_b >0$, this implies
\begin{align*}
&C_i\ge A_j\ge D_{ijj},\text{ and}\\
&C_i\ge B_j> D_{jij}.
\end{align*}
Thus applying Lemma \ref{lem:convex} to equation (\ref{eq:ddabc}),
with
$$a=A_j,\ b=B_j,\ c=C_i,\ d_1=D_{ijj},\ d_2=D_{jij},$$ we find
$\beta$ is uniquely determined, so $\alpha=1/\beta$ is identifiable.

Once $\alpha$ is known, for every $k$ we may determine the
quantities
\begin{align*}
\lambda_k(t_a+t_b)=L_\alpha^{-1}(C_k),\\
\lambda_k(t_a+t_c)=L_\alpha^{-1}(B_k),\\
\lambda_k(t_b+t_c)=L_\alpha^{-1}(A_k).
\end{align*}
Thus we may determine the ratio between any two eigenvalues
$\lambda_k$. As $U$ is known, this determines $Q$ up to scaling.
Since we have required a normalization of $Q$, this means $Q$ is
identifiable. With the $\lambda_k$ now determined, we can find
$t_a+t_b$, $t_a+t_c$ and $t_b+t_c$, and hence $t_a,t_b,t_c$.
\end{proof}

\section{Exceptional cases ($\kappa=4$)}\label{sec:badmods}

In the previous section, identifiability was proved under the
assumption that $\nu_{ijj}\neq 0$ for some $j\ge i>1$. We now
specialize to the case of $\kappa=4$, and determine those GTR
parameters for which none of these conditions holds. In the
subsequent section, we will use this information to argue that even
in these exceptional cases the GTR+$\Gamma$ model is identifiable.

Note that while we work only with a 4-state model appropriate to
DNA, the approach we use may well apply for larger $\kappa$, though
one should expect additional exceptional subcases to appear.

\medskip
\begin{lemma}\label{lem:except}
For $\kappa=4$, consider a choice of GTR parameters for which
$\nu_{ijj}=0$ for all $j\ge i>1$. Then, up to permutation of the
states and multiplication of some columns of $U$ by $-1$, the
distribution vector $\boldsymbol \pi$ and eigenvector matrix $U$
satisfy one of the two following sets of conditions:

\smallskip

\textbf{Case A:} $\boldsymbol \pi=(1/4,\,1/4,\,1/4,\,1/4)$, and for
some $b,c\ge 0$ with $b^2+c^2=2$,
$$U=\begin{pmatrix}
1&\phantom{-}c&\phantom{-}b&\phantom{-}1\\1&-c&-b&\phantom{-}1\\
1&-b&\phantom{-}c&-1\\1&\phantom{-}b&-c&-1
\end{pmatrix}$$

\smallskip

\textbf{Case B:} $\boldsymbol \pi=(1/8,\,1/8,\,1/4,\,1/2)$, and
$$U=\begin{pmatrix}
1&\phantom{-}2&\phantom{-}\sqrt2&\phantom{-}1\\1&-2&\phantom{-}\sqrt2&\phantom{-}1\\
1&\phantom{-}0&-\sqrt2&\phantom{-}1\\1&\phantom{-}0&\phantom{-}0&-1
\end{pmatrix}$$

\end{lemma}

\begin{proof}
We use the notation of Proposition \ref{prop:non0}, including the
inner product and definition of vectors $W_{ij}$ given in its proof.
Orthogonality and lengths will always be with respect to that inner
product.

We will repeatedly use that for $i,j$ with $1<i\le j$,
$$\langle
W_{jj}, U_i \rangle=\nu_{ijj}=0.$$

In particular, setting $j=4$, we find $W_{44}$ is orthogonal to
$U_2,U_3,U_4$, and hence is a multiple of $U_1=\one$. This implies
$$U_4=(\pm 1,\pm1 ,\pm1, \pm1),$$ since $U_4$ has length 1. Without
loss of generality, by possibly permuting the rows of $U$ (which is
equivalent to changing the ordering of the states in writing down
the rate matrix $Q$), and then possibly multiplying $U_4$ by $-1$,
we need now only consider two cases: either
\begin{align*}
&\text{Case A: }\ \ \ \  U_4=(1,1,-1,-1), \text{ or}\\
&\text{Case B: }\ \ \ \  U_4=(1,1,1,-1).
\end{align*}
We consider these two cases separately.
\smallskip

\noindent \textbf{Case A:} Since $U_1=\one$ and $U_4=(1,1,-1,-1)$,
the orthogonality of $U_1$ and $U_4$ gives
$$\pi_1+\pi_2-\pi_3-\pi_4=0.$$
Since $\sum_{i=1}^4 \pi_i=1$, this tells us
\begin{equation}
\pi_1+\pi_2=1/2,\ \ \pi_3+\pi_4=1/2.\label{eq:psum}
\end{equation}

Now since $W_{33}$ is orthogonal to both $U_2$ and $U_3$, then
$W_{33}$ is a linear combination of $U_1$ and $U_4$, and hence
$W_{33}=(b^2,b^2,c^2,c^2)$. Thus
$$U_3=(\pm b, \pm b, \pm c,\pm c).$$
Since $U_3$ is orthogonal to both $U_1$ and $U_4$, it is orthogonal
to their linear combinations, and in particular to $(1,1,0,0)$ and
$(0,0,1,1)$. Thus, by permuting the first two entries of the $U_i$,
and also permuting the last two entries of the $U_i$, if necessary,
we may assume
$$U_3=(b,-b,c,-c)$$ with $b,c\ge 0$. This orthogonality further shows
$$b\pi_1-b\pi_2=0,\ \ c\pi_3-c\pi_4=0.$$
Thus $$\pi_1=\pi_2,\ \ \text{or \ } b=0,$$ and
$$\pi_3=\pi_4,\ \ \text{or \ } c=0.$$
In light of equations (\ref{eq:psum}), we have
$$\pi_1=\pi_2=1/4,\ \ \text{or \ } b=0,$$ and
$$\pi_3=\pi_4=1/4,\ \ \text{or \ } c=0.$$

\

In any of these cases,  $U_3$ has length 1 so
$$b^2(\pi_1+\pi_2)+c^2(\pi_3+\pi_4)=1.$$
Together with equations (\ref{eq:psum}) this gives that
$$b^2+c^2=2.$$

Now since $U_2$ is orthogonal to $U_1,U_3,U_4$, we must have that
$$U_2=a(c/\pi_1,-c/\pi_2, -b/\pi_3,b/\pi_4)$$ for some $a$, and we
may assume $a>0$. But the length of $U_2$ is 1, and $U_2$ is
orthogonal to $W_{22}$, so
\begin{align} &c^2/\pi_1
+c^2/\pi_2+b^2/\pi_3+b^2/\pi_4=1/a^2,\label{eq:f1}\\
&c^3/\pi_1^2-c^3/\pi_2^2-b^3/\pi_3^2+b^3/\pi_4^2=0.\label{eq:f2}
\end{align}

If neither of $b,c$ is zero, so all $\pi_i=1/4$, then  equation
(\ref{eq:f1}) tells us $a=1/4$, as the statement of the theorem
claims.

If $b=0$, then we already know $c=\sqrt2$, and $\pi_3=\pi_4=1/4$.
But equation (\ref{eq:f2}) implies $\pi_1=\pi_2$, so these are also
$1/4$. We then find from equation (\ref{eq:f1}) that $a=1/4$, and we
have another instance of the claimed characterization of case A.
Similarly, if $c=0$ we obtain the remaining instance.

\medskip

\noindent \textbf{Case B:} Since $U_1=\one$ and $U_4=(1,1,1,-1)$,
 the orthogonality of $U_1$ and $U_4$ implies
$$\pi_1+\pi_2+\pi_3-\pi_4=0.$$

Now $W_{33}$ is orthogonal to $U_2$ and $U_3$, and hence is a linear
combination of $U_1$ and $U_4$. Thus $W_{33}=(b^2,b^2,b^2,c^2),$ so
$$U_3=(\pm b, \pm b ,\pm b, c).$$ But $U_3$ is orthogonal to both
$U_1$ and $U_4$, and hence orthogonal to their linear combinations,
including $(0,0,0,1)$ and $(1,1,1,0)$. This shows $c=0$ and that
(possibly by permuting the first three rows of $U$, and multiplying
$U_3$ by $-1$) we may assume $U_3=b(1,1,-1,0)$ for some $b>0$.
Orthogonality of $U_3$ and $U_1$ then shows
$$\pi_1 +\pi_2 -\pi_3 =0.$$

Also $W_{22}$ is orthogonal to $U_2$, and hence is a linear
combination of $U_1,U_3,U_4$, so $W_{22}=(d^2,d^2,e^2,f^2)$. Thus
$$U_2=(\pm d,\pm d,e, f).$$ However, since $U_2$ is orthogonal to
$U_1,U_3,U_4$, it is orthogonal to $(0,0,0,1)$, $(0,0,1,0)$, and
$(1,1,0,0)$. Thus we may assume $U_2=d(1,-1,0,0)$ with $d>0$.
Finally, orthogonality of $U_2$ and $U_1$ implies
$$\pi_1-\pi_2=0.$$

All the above equations relating the $\pi_i$, together with the fact
that $\sum_{i=1}^4\pi_i=1$ gives
$$\boldsymbol \pi=\pi_1(1,\,1,\,2,\,4)=(1/8,\,1/8,\,1/4,\,1/2).$$

We can now determine the $U_i$ exactly, using that they must have
length 1, to show $U$ is as claimed.

\end{proof}

\section{Identifiability in the exceptional cases
($\kappa=4$)}\label{sec:badid}

We now complete the proof of Theorem \ref{thm:id} by showing
identifiability in cases A and B of Lemma \ref{lem:except}. We do
this by first establishing some inequalities for the eigenvalues of
$Q$ that must hold in each of these cases, using the assumption that
the off-diagonal entries of $Q$ are positive.

Note that as $U^{-1}=U^T\diag(\boldsymbol\pi)$, and the entries of
$\boldsymbol \pi$ are positive, the positivity of the off-diagonal
entries of $Q$ is equivalent to the positivity of the off-diagonal
entries of the symmetric matrix
$$\widetilde Q=U\diag(0,\lambda_2,\lambda_3,\lambda_4)U^T.$$

\begin{lemma}\label{lem:Aineq} For $\kappa=4$, let
$0=\lambda_1>\lambda_2\ge \lambda_3\ge \lambda_4$
denote the eigenvalues of a GTR rate matrix $Q$. Then the following
additional inequalities hold in cases A and B of Lemma
\ref{lem:except}:

\smallskip

Case A: If \  $bc\ne0$, then $\lambda_4>\lambda_2+\lambda_3,$ while
if \ $bc=0$, then $\lambda_4> 2\lambda_2$.

\smallskip

Case B: $\lambda_4>2\lambda_2. $
\end{lemma}

\begin{proof} For case A, one computes that
$$\widetilde Q =
\begin{pmatrix}
*&-\lambda_2c^2-\lambda_3b^2+\lambda_4&-\lambda_2bc+\lambda_3
bc-\lambda_4& \lambda_2 bc -\lambda_3 bc -\lambda_4\\
*&*&\lambda_2bc-\lambda_3bc-\lambda_4&-\lambda_2bc+\lambda_3bc-\lambda_4\\
*&*&*&-\lambda_2b^2-\lambda_3c^2+\lambda_4\\ *&*&*&*
\end{pmatrix}
$$ where the stars indicate quantities not of interest. From the positivity of the (1,2) and (3,4) entries of
$\widetilde Q$, we thus know
$$\lambda_4>\max(\lambda_2c^2+\lambda_3b^2,
\lambda_2b^2+\lambda_3c^2)\ge\frac {(\lambda_2c^2+\lambda_3b^2)+
(\lambda_2b^2+\lambda_3c^2)}2.$$ Since $b^2+c^2=2$, this shows
$\lambda_4>\lambda_2+\lambda_3.$ In the case when $bc=0$, so
$(b,c)=(0,\sqrt 2)$ or $(\sqrt 2,0)$, the first inequality gives the
stronger statement of the proposition.

For case B,
$$ \widetilde Q =\begin{pmatrix}
*&-4\lambda_2+2\lambda_3+\lambda_4&-2\lambda_3
+\lambda_4&-\lambda_4\\
*&*&-2\lambda_3+\lambda_4&-\lambda_4\\
*&*&*&-\lambda_4
\\ *&*&*&*
\end{pmatrix}
$$
From the positivity of the off-diagonal entries, we see that
$$\lambda_4>2\lambda_3,\ \ \lambda_4+2\lambda_3>4\lambda_2.$$
Together, these imply that $\lambda_4>2\lambda_2$.
\end{proof}

\bigskip

We now return to proving identifiability for the exceptional cases.
As in the proof of Theorem \ref{thm:genid}, we may determine the
relative rankings of $t_a$, $t_b$ and $t_c$, and therefore assume
$$0\le t_a\le t_b\le t_c,$$
with $t_b>0$.
\medskip

In case A, we find that $\nu_{234}=bc$, so we break that case into
two subcases, \begin{align*} &\text{Case A1: if $b,c\ne 0$; and}\\
&\text{Case A2: if $b$ or $c=0$.}
\end{align*}

\smallskip

\noindent \textbf{Case A1:} In this case, we find that $\nu_{ijk}\ne
0$ for all distinct $i,j,k>1$. Letting
\begin{align*}
D_{342}&=L_\alpha(\lambda_3 t_a+\lambda_4 t_b+\lambda_2 t_c),\\
D_{423}&=L_\alpha(\lambda_4 t_a+\lambda_2 t_b+\lambda_3 t_c).
\end{align*}
and $A_k,B_k,C_k$ be as in the proof of Theorem \ref{thm:genid},
observe that
\begin{equation*}
L_\alpha^{-1}(D_{342})+L_\alpha^{-1}(D_{423})=
L_\alpha^{-1}(A_2)+L_\alpha^{-1}(B_3)+L_\alpha^{-1}(C_4).
\end{equation*}
Setting $\beta=1/\alpha$ and using the explicit formula for
$L_\alpha$ yields
\begin{equation}
D_{342}^{-\beta}+D_{423}^{-\beta}-A_2^{-\beta}-B_3^{-\beta}-C_4^{-\beta}+1=0.
\label{eq:JClike}\end{equation} Note that by Proposition
\ref{prop:triple} all constants in this equation, except possibly
$\beta$, are uniquely determined by the joint distribution.

In preparation for applying Lemma \ref{lem:convex}, we claim that
the following inequalities hold:
\begin{align} D_{342}\le A_2,\label{ineq:1}\\
D_{423}<B_3,\label{ineq:2}\\ D_{342}<C_4,\label{ineq:3}\\
D_{423}<C_4.\label{ineq:4}
\end{align}
Inequalities (\ref{ineq:1},\ref{ineq:2}) follow easily from the fact
that $L_\alpha$ is increasing. For inequality (\ref{ineq:3}), note
first that $\lambda_3t_a+\lambda_4t_b+\lambda_2t_c\le
(\lambda_2+\lambda_3)t_a+\lambda_4t_b$. But Lemma \ref{lem:Aineq}
indicates $\lambda_2+\lambda_3<\lambda_4$, so, again using that
$L_\alpha$ is increasing, the claim follows. Inequality
(\ref{ineq:4}) is similarly shown to hold.

Finally, to apply Lemma \ref{lem:convex}, let $d_1=D_{342}$,
$d_2=D_{423}$. The remainder of the constants in the lemma are
chosen in one of three ways, depending on which of $A_2,B_3,C_4$ is
largest:

 If $C_4\ge A_2, B_3$, then let $a=A_2$, $b=B_3$, $c=C_4$.

 If $A_2\ge C_4, B_3$, then let $a=C_4$, $b=B_3$, $c=A_2$.

 If $B_3\ge C_4, A_2$, then let $a=A_2$, $b=C_4$, $c=B_3$.

\noindent Thus in all subcases, from equation (\ref{eq:JClike}) we
find that $\beta>0$ is uniquely determined.

The remainder of the proof now proceeds exactly as for Theorem
\ref{thm:genid}.

\smallskip

\noindent \textbf{Cases A2 and B:} In both of these cases
$\nu_{224}\ne 0$, so, similarly to the previous case, letting
\begin{align*}
D_{422}&=L_\alpha(\lambda_4t_a+\lambda_2 t_b+\lambda_2t_c),\\
D_{242}&=L_\alpha(\lambda_2t_a+\lambda_4 t_b+\lambda_2t_c),
\end{align*}
leads to
\begin{equation}
D_{422}^{-\beta}+D_{242}^{-\beta}-C_4^{-\beta}-A_2^{-\beta}-B_2^{-\beta}
+1=0. \label{eq:notJClike}
\end{equation}
By Proposition \ref{prop:triple}, we know all quantities in this
equation except possibly $\beta$ are uniquely determined from the
joint distribution.

We also note the following inequalities hold:
\begin{align}
A_2&\le B_2,\label{ineq:5}
\\D_{422}&\le A_2,\label{ineq:6}
\\D_{242}&<B_2,\label{ineq:7}
\\D_{242}&<C_4.\label{ineq:8}
\end{align}
Inequalities (\ref{ineq:5}--\ref{ineq:7}) are implied by the fact
the $L_\alpha$ is increasing. Inequality (\ref{ineq:8}) will follow
from $\lambda_2(t_a+t_c)<\lambda_4 t_a$. However,
$\lambda_2(t_a+t_c)\le 2\lambda_2t_a<\lambda_4 t_a$ by Lemma
\ref{lem:Aineq}.

To apply Lemma \ref{lem:convex}, let $d_1=D_{422}$ and
$d_2=D_{242}$. In light of inequality (\ref{ineq:5}), we need assign
the remaining constants according to only two cases:

If $C_4\ge B_2$, let $a=A_2$, $b=B_2$, and $c=C_4$.

If $B_2 \ge C_4$, let $a=A_2$, $b=C_4$, and $c=B_2$.

\noindent In both cases, we find $\beta$ is uniquely determined, and
the the proof of identifiability can be completed as in Theorem
\ref{thm:genid}.

\medskip

Thus identifiability of the GTR+$\Gamma$ model when $\kappa=4$ is
established for all cases.

\section{Open problems}\label{sec:open}

Many questions remain on the identifiability of phylogenetic models,
including those commonly used for data analysis.

Perhaps the most immediate one is the identifiability of the
GTR+$\Gamma$+I model. Despite its widespread use in inference, no
proof has appeared that the tree topology is identifiable for this
model, much less its numerical parameters. Although our algebraic
arguments of Section \ref{sec:alg} apply, analogs for GTR+$\Gamma$+I
of the analytic arguments we gave for GTR+$\Gamma$ are not obvious.
While the $\Gamma$ rate distribution has only one unknown parameter,
$\Gamma$+I has two, and this increase in dimensionality seems to be
at the heart of the difficulty. Interestingly, empirical studies
\cite{SSN99} have also shown that these parameters can be difficult
to tease apart, as errors in their inferred values can be highly
correlated in some circumstances. Although we conjecture that
GTR+$\Gamma$+I is identifiable for generic parameters, we make no
guess as to its identifiability for all parameters.

For computational reasons, standard software packages for
phylogenetic inference implement a discretized $\Gamma$ distribution
\cite{Yang94}, rather than the continuous one dealt with in this
paper. While results on continuous distributions are suggestive of
what might hold in the discrete case, they offer no guarantee. It
would therefore also be highly desirable to have proofs of the
identifiability of the discretized variants of GTR+$\Gamma$ and
GTR+$\Gamma$+I, either for generic or all parameters. Note that such
results might depend on the number of discrete rate classes used, as
well as on other details of the discretization process. So far the
only result in this direction is that of \cite{ARidtree} on the
identifiability of the tree parameter, for generic numerical
parameter choices when the number of rate classes is less than the
number of observable character states (\emph{e.g.}, at most 3 rate
classes for 4-state nucleotide models, or at most 60 rate classes
for 61-state codon models). As the arguments in that work use no
special features of a $\Gamma$ distribution, or even of an
across-site rate variation model, we suspect that stronger claims
should hold when specializing to a particular form of a discrete
rate distribution.

Finally, we mention that beyond \cite{ARidtree}, almost nothing is
known on identifiability of models with other types of
heterogeneity, such as covarion-like models and general mixtures. As
these are of growing interest for addressing biological questions
\cite{GasGui,PagMea,KT}, much remains to be understood.

\acks{ESA and JAR thank the Institute for Mathematics and Its
Applications and the Isaac Newton Institute, where parts of
this work were undertaken, for their hospitality and funding.
Work by ESA and JAR was also supported by the National Science
Foundation (DMS 0714830).}

\bibliographystyle{plain}

\bibliography{Phylo}

\newpage

\appendix

\section{The gaps in Rogers' proof}

Here we explain the gaps in the published proof of Rogers
\cite{Rog01} that the GTR+$\Gamma$+I model is identifiable. Since
that paper has been widely cited and accepted as correct, our goal
is to clearly indicate where the argument is flawed, and illustrate,
through some examples, the nature of the logical gaps.

We emphasize that we do not \emph{prove} that the gaps in the
published argument cannot be bridged. Indeed, it seems most likely
that the GTR+$\Gamma$+I model is identifiable, at least for generic
parameters, and it is possible a correct proof might follow the
rough outline of \cite{Rog01}. However, we have not been able to
complete the argument Rogers attempts. Our own proof of the
identifiability of the GTR+$\Gamma$ model presented in the body of
this paper follows a different line of argument.

We assume the reader of this appendix will consult \cite{Rog01}, as
pinpointing the flaws in that paper requires rather technical
attention to the details in it.

\subsection{Gaps in the published proof}

There are two gaps in Rogers' argument which we have identified. In
this section we indicate the locations and nature of these flaws,
and in subsequent ones we elaborate on them individually.

The first gap in the argument occurs roughly at the break from page
717 to page 718 of the article. To explain the gap, we first outline
Rogers' work leading up to it. Before this point, properties of the
graph of the function $\nu^{-1}(\mu(x))$ have been carefully
derived. An example of such a graph, for particular values of the
parameters $\alpha,a,\pi,p$ occurring in the definitions of $\nu$
and $\mu$, is shown in Figure 2 of the paper. For these parameter
values and others, the article has carefully and correctly shown
that for $x\ge 0$ the graph of $\nu^{-1}(\mu(x))$
\begin{enumerate}
\item is increasing,
\item has a single inflection point, where the graph changes from
convex to concave (i.e, the concavity changes from upward to
downward),
\item has a horizontal asymptote as $x\to \infty$.
\end{enumerate}
Although the article outlines other cases for different ranges of
the parameter values, Rogers highlights the case when these three
properties hold.

At the top of page 718 of the article, Figure 3 is presented,
plotting the points whose coordinates are given by the pairs
$(\nu^{-1}(\mu(\tau_1\lambda_i)),\nu^{-1}(\mu(\tau_2\lambda_i)))$
for all $\lambda_i\ge 0$. Here $\tau_2>\tau_1$ are particular
values, while $\alpha,a, \pi,p$ are given the values leading to
Figure 2. Rogers points out that ``As in Figure 2, the graph [of
Figure 3] has an inflection point, is concave upwards before the
inflection point, and is concave downwards after the inflection
point.'' Then he claims that ``Similar graphs will be produced for
any pair of path distances such that $\tau_2>\tau_1$.'' However, he
gives no argument for this claim. As the remainder of the argument
strongly uses the concavity properties of the graph of his Figure 3
(in the second column on page 718 the phrase ``\dots as shown by
Figure 3'' appears), without a proof of this claim the main result
of the paper is left unproved.

Judging from the context in which it is placed, a more complete
statement of the unproved claim would be that for any values of
$\alpha,a,\pi,p$ resulting in a graph of $\nu^{-1}(\mu(x))$ with the
geometric properties of Figure 2, and any $\tau_2>\tau_1$, the graph
analogous to Figure 3 has a single inflection point. As no argument
is given to establish the claim, we can only guess what the author
intended for its justification. From what appears earlier in the
paper, it seems likely that the author believed the three geometric
properties of the graph in Figure 2 enumerated above implied the
claimed properties of Figure 3. However, that is definitely not the
case, as we will show in Section \ref{sec:graph} below.

Note that we do not assert that the graphs analogous to Figure 3 for
various parameter values are not as described in \cite{Rog01}. While
plots of them for many choices of parameter values certainly suggest
that Rogers' claim holds, it is of course invalid to claim a proof
from examples. Moreover, with 4 parameters $\alpha,a,\pi,p$ to vary,
it is not clear how confident one should be of even having explored
the parameter space well enough to make a solid conjecture. In light
of the example we give in Section \ref{sec:graph}, justifying
Rogers' claim would require a much more detailed analysis of the
functions $\nu$ and $\mu$ than Rogers attempts.

\medskip

If this first gap in the proof were filled, a second problem would
remain. Though less fundamental to the overall argument, this gap
would mean that identifiability of the model would be established
for \emph{generic} parameters, but that there might be exceptional
choices of parameters for which identifiability failed. (`Generic'
here can be taken to mean for all parameters except those lying in a
set of Lebesgue measure zero in parameter space. More informally,
for any reasonable probability distribution placed on the parameter
space, randomly-chosen parameters will be generic.)

Although the origin of this problem with non-generic parameters is
clearly pointed out by Rogers, it is open to interpretation whether
he attempts to extend the proof to all parameter values at the very
end of the article. However, as the abstract and introductory
material of \cite{Rog01} make no mention of the issue, this point at
the very least seems to have escaped many readers attention.

This gap occurs because the published argument requires that the
non-zero eigenvalues of the GTR rate matrix $Q$ be three
\emph{distinct} numbers. On page 718, at the conclusion of the main
argument, it is stated that ``Therefore, if the substitution rate
matrix has three distinct eigenvalues, the parameters of the
I+$\Gamma$ rate heterogeneity will be uniquely determined.'' The
author then goes on to point out that for the Jukes-Cantor and
Kimura 2-parameter models this assumption on eigenvalues is
violated, but ``[f]or real data sets, however, it is unlikely that
any two or all three of the eigenvalues will be exactly identical.''

Leaving aside the question of what parameters one might have for a
model which fits a real data set well, Rogers here clearly indicates
that his proof of identifiability up to this point omits some
exceptional cases. In the concluding lines of the paper, he points
out that these exceptional cases can be approximated arbitrarily
closely by parameters with three distinct eigenvalues. While this is
true, such an observation cannot be used to argue that the
exceptional cases are not exceptional, as we will discuss below in
Section \ref{sec:generic}. It is unclear whether the concluding
lines of \cite{Rog01} were meant to `fill the gap' or not.

Of course, one might not be too concerned about exceptional cases.
Indeed, if the first flaw were not present in his argument, then
Rogers' proof would still be a valuable contribution in showing that
for `most' parameter values identifiability held. One might then
look for other arguments to show identifiability also held in the
exceptional cases.  Nonetheless, it is disappointing that the
exceptional cases include models such as the Jukes-Cantor and Kimura
2-parameter that are well-known to biologists and might be
considered at least reasonable approximations of reality in some
circumstances.

\subsection{A counterexample to the graphical
argument}\label{sec:graph}

It seems that the origin of the first flaw in Rogers' argument is in
a belief that the three enumerated properties he proves are
exhibited in his Figure 2 result in the claimed properties of his
Figure 3. In this section, we show this implication is not valid, by
exhibiting a function whose graph has the three properties, but when
the graph analogous to Figure 3 is constructed, it has multiple
inflection points.

\smallskip

Let $$f(x)=\int_0^x
\exp\left(\exp\left(-10(t-1)^2\right)-\frac{(1-t)^2}{10}\right)\,dt.$$
Then $f(0)=0$, and
$$f'(x)=\exp\left(\exp\left(-10(x-1)^2\right)-\frac{(1-x)^2}{10}\right),$$ so
$f'(x)>0$ and $f$ is increasing. Furthermore, one sees that $f'(x)$
decays quickly enough to 0 as $x\to \infty$, so that $f(x)$ has a
horizontal asymptote as $x\to \infty$.

To see that $f(x)$ has a single inflection point where the graph
passes from convex to concave, it is enough to show $f'(x)$ has a
unique local maximum and no local minima. But this would follow from
$g(x)=\ln(f'(x))$ having a unique local maximum and no local minima.
Since
$$g(x)=\exp\left(-10(x-1)^2\right)-\frac{(1-x)^2}{10},$$ and the two summands
here have unique local maxima at $x=1$ and no local minima, $g$ must
as well. Thus $f$ exhibits the enumerated properties of Rogers'
Figure 2. For comparison, we graph $f$ in our Figure \ref{fg:graphf}
below.
\begin{figure}[h]
\begin{center}
\includegraphics[height= 1.in]{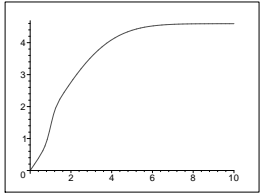}
\end{center}
\caption{The graph $y=f(x)$.}\label{fg:graphf}
\end{figure}

\medskip

The analog of Figure 3 for the function $f$ would show the points
$(f(\tau_1 x),f(\tau_2 x))$. If we choose $\tau_1=1,\tau_2=2$, we
obtain the graph shown in our Figure \ref{fg:badgraph}.
\begin{figure}[h]
\begin{center}
\includegraphics[height= 1.in]{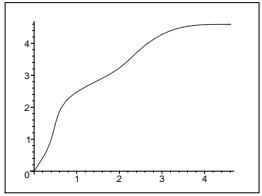}
\end{center}
\caption{The points $(f(x),f(2x))$.}\label{fg:badgraph}
\end{figure}
Obviously, the curve in Figure \ref{fg:badgraph} has multiple
--- at least three --- inflection points. Although we will not give a
formal proof here that this curve has multiple inflection points, it
is not difficult to do so.

\medskip

\subsection{Identifiability for generic parameters vs.~all parameters
}\label{sec:generic}

The second gap in Rogers' argument arises because it is possible to
have identifiability for generic parameters, but not for all
parameters. Even if identifiability of generic parameters has been
proved, then one cannot easily argue that identifiability must hold
for the non-generic, exceptional cases as well. To illustrate this,
we give a simple example.

Consider the map $\phi:\mathbb R^2\to \mathbb R^2$, defined by
$$\phi(a,b)=\left(a,\,ab\right).$$
Here $a,b$ play the roles of `parameters' for a hypothetical model,
whose `joint distribution' is given by the vector-valued function
$\phi$.

Suppose $(x,y)$ is a particular distribution which arises from the
model (\emph{i.e.}, is in the image of $\phi$), and we wish to find
$a,b$ such that $\phi(a,b)=(x,y)$. Then provided $x\ne 0$ (or
equivalently $a\ne 0$), it is straightforward to see that $a,b$ must
be given by the formulas
$$a=x,\ \ b=y/x.$$
Thus for generic $a,b$ (more specifically, for all $(a,b)$ with
$a\ne0$) this hypothetical model is identifiable.

Notice, however, that if $(x,y)=(0,0)$, the situation is quite
different. From $x=0$, we see that we must have $a=0$. But since
$\phi(0,b)=(0,0)$, we find that all parameters of the form $(0,b)$
lead to the same distribution $(0,0)$. Thus these exceptional
parameters are \emph{not} identifiable. Therefore, we have
identifiability precisely for all parameters in the 2-dimensional
$ab$-plane \emph{except} those lying on the 1-dimensional line where
$a=0$. These exceptional parameters, forming a set of lower
dimension than the full space, have Lebesgue measure zero within it.

Notice that even though there are parameter values arbitrarily close
to the exceptional ones $(0,b)$ which are identifiable (for
instance, $(\epsilon,b)$ for any small $\epsilon\ne 0$), it is
invalid to argue that the parameters $(0,b)$ must be identifiable as
well.

This example shows that even if the first flaw in the argument of
\cite{Rog01} were repaired, the approach outlined there will at best
give identifiability for generic parameters. The final lines of that
paper are not sufficient to prove identifiability for all parameter
values.

Obviously the function $\phi$ given here could not really be a joint
distribution for a statistical model, since the entries of the
vector $\phi(a,b)$ do not add to one, nor are they necessarily
non-negative. However, these features can be easily worked into a
more complicated example. If one prefers a less contrived example,
then instances of generic identifiability of parameters but not full
identifiability occur in standard statistical models used outside of
phylogenetics (for instance, in latent class models). We have chosen
to give this simpler example to highlight the essential problem most
clearly.

\end{document}